\documentclass[a4paper,10pt]{amsart}

\usepackage[english]{babel}

\usepackage{lmodern}

\usepackage[utf8]{inputenc}

\usepackage[babel=true]{csquotes}

\usepackage{comment}

\usepackage{amsfonts}
\usepackage{amsthm} 
\usepackage{amsmath}
\usepackage{amssymb}

\usepackage{multirow} 
\usepackage{kbordermatrix}

\usepackage{brunnian}

\usetikzlibrary{arrows}
\usetikzlibrary{matrix}

\theoremstyle{definition}
\newtheorem{defi}{Definition}[section] 

\theoremstyle{plain}
\newtheorem{lemma}[defi]{Lemma} 
\newtheorem{prop}[defi]{Proposition}
\newtheorem{theo}[defi]{Theorem}
\newtheorem{corollary}[defi]{Corollary}

\newtheorem*{maintheo}{Main Theorem}
\newtheorem*{theo41}{Theorem \ref{thm 41}}
\newtheorem*{theo52}{Theorem \ref{thm 52}}

\theoremstyle{remark}
\newtheorem{remark}[defi]{Remark}
\newtheorem{exemple}[defi]{Example}

\newcommand{\C}{\mathbb{C}}
\newcommand{\R}{\mathbb{R}}

\newcommand{\Z}{\mathbb{Z}}
\newcommand{\N}{\mathbb{N}}
\newcommand{\F}{\mathbb{F}}

\newcommand{\NN}{\mathcal{N}}

\title[The $L^2$-Alexander invariant is stronger than (genus,volume)]{The $L^2$-Alexander invariant is stronger than the genus and the simplicial volume}
\author{Fathi Ben Aribi}

\address{Universit\'e de Gen\`eve, Section de math\'ematiques, 2-4 rue du Li\`evre, Case postale 64 1211 Gen\`eve 4, Suisse}

\email{fathi.benaribi@unige.ch}
\begin{document}

\renewcommand{\proofname}{Proof}


\maketitle

\begin{abstract}
We study how the genus, the simplicial volume and the $L^2$-
Alexan-der 
invariant of W. Li and W. Zhang can detect individual knots among all others. In particular, we use various techniques coming from hyperbolic geometry and topology to prove that the $L^2$-Alexander invariant contains strictly more information than the pair (genus, simplicial volume). Along the way we prove that the $L^2$-Alexander invariant
detects the figure-eight knot $4_1$, the twist knot $5_2$ and an infinite family of cables on the figure-eight knot.

\end{abstract}

\section{Introduction}
The strength of a knot invariant is inexorably tied to how well it can distinguish two given knots. In particular, invariants that can detect individual knots among all others hold great interest. 

The $L^2$-Alexander invariant $\Delta_K^{(2)}$ of a knot $K$ is a knot invariant taking values in the class of maps on the positive real numbers up to multiplication by monomials. It was originally constructed by W. Li and W. Zhang in \cite{LZ06a} from a Wirtinger presentation of the knot group, as an infinite-dimensional version of Fox's construction of the Alexander polynomial (see \cite{fox}). Li and Zhang proved that the value $\Delta_K^{(2)}(t=1)$ of the invariant equals the $L^2$-torsion of the exterior of the knot $K$, which is exactly given by the simplicial volume $\mathrm{vol}(K)$ from a theorem of W. Lück and T. Shick (see \cite{LS99}).

J. Dubois and C. Wegner then proved in \cite{DW} that one could compute the $L^2$-Alexander invariant from any deficiency one presentation of the knot group. They also computed that the invariant of a torus knot $K$ is equal to $(t \to \max(1,t))$ to the power twice the genus $g(K)$ of the torus knot.

In \cite{BA13}, the author generalized this property to the class of iterated torus knots by proving a connected sum formula and a cabling formula for the $L^2$-Alexander invariant. As a consequence, the author proved that the invariant detects the unknot, as well as the trefoil knot $3_1$ up to mirror image in \cite{BAthesis}.

J. Dubois, S. Friedl and W. Lück then proved in \cite{DFLdual, DFL} that the $L^2$-Alexander invariant is always symmetric and that in the case of a fibered knot, its extreme values are monomial with span of degree twice the genus of the knot.

Recently, Friedl-Lück and Y. Liu proved in \cite{FL,Liu} that the $L^2$-Alexander invariant $\Delta_K^{(2)}$ detects the genus $g(K)$ of any knot $K$, fibered or not, by comparing the behaviours of $\Delta_K^{(2)}(t)$ at $t\to 0^+$ and $t \to + \infty$.

\bigskip

It is natural to ask how many knots can be individually detected by the $L^2$-Alexander invariant, and whether or not the invariant is reduced to the pair (genus, simplicial volume) it was proven to contain. 
The Main Theorem of this article answers simultaneously both these questions, respectively by ``an infinite number'' and ``the $L^2$-Alexander invariant is strictly stronger''.
%
\begin{maintheo}[Theorem \ref{thm inf}]
There exists two infinite families of fibered knots $(J_n)_{n  \in \N}$ and $(K_n)_{n  \in \N}$ such that for all $n \in \N$ :
\begin{enumerate}
\item $g(J_n) = g(K_n)$,
\item $\mathrm{vol}(J_n) = \mathrm{vol}(K_n)$,
\item $\Delta_{J_n}^{(2)} \neq \Delta_{K_n}^{(2)}$,
\item $K_n$ is the only knot (up to mirror image and change of orientation) with $L^2$-Alexander invariant equal to $\Delta_{K_n}^{(2)}$.
\end{enumerate}
\end{maintheo}
The knots $J_n$ are explicitily constructed as connected sums of one figure-eight knot and several trefoil knots, whereas the knots $K_n$ are cables of the figure-eight knot. They have same genus and same volume for each value $n$, but different $L^2$-Alexander invariants.
This theorem thus establishes that the $L^2$-Alexander invariant contains \textit{strictly more information} than the pair (genus, simplicial volume), and detects an \textit{infinite number} of knots.

\bigskip

As a stepping stone to the Main Theorem, we prove that the
pair (genus,volume), and by extension the $L^2$-Alexander invariant, both detect the figure-eight knot:
\begin{theo41}
The $L^2$-Alexander invariant detects the figure-eight knot $4_1$.
\end{theo41}
The proof of Theorem \ref{thm 41} can be decomposed into two parts. We provide some of the details because the techniques involved are also used in the proof of the Main Theorem.

First, we use the fact that for a fibered knot, 
the extremes values of the $L^2$-Alexander invariant are monomial, with degrees of span equal to twice the genus. This was proven by J. Dubois, S. Friedl and W. Lück in \cite{DFL}, and we present in Theorem \ref{thm L2 alex fibered} a new proof of this result relying directly on Fox calculus. We also find a new bound on the monomiality limit depending on operator norms of the Fox jacobian associated to the monodromy of the knot.

The second part of the proof relies on properties of hyperbolic volumes of link exteriors proven by Cao-Meyerhoff, Adams and Agol in \cite{adams1988volumes,agol2010minimal,cao} that guarantee that a knot with simplicial volume $2.029..$ is fibered and either $4_1$ or obtained from $4_1$ from operations that strictly increase the genus, which is not possible since the genus one is detected by the $L^2$-Alexander invariant thanks to the first part.

\bigskip

One can ask if only fibered knots can be detected by the $L^2$-Alexander invariant. We answer by the negative with the following theorem:
\begin{theo52}
The $L^2$-Alexander invariant detects the (non-fibered) knot $5_2$ up to mirror image.
\end{theo52}
This time we must use the more general fact that the genus is always detected even if the knot is not fibered, proven by  Friedl-Lück and Liu in \cite{FL, Liu}, as well as properties of small hyperbolic volumes proven by Gabai-Meyerhoff-Milley in \cite{GMM}. Like for the figure-eight knot, the pair (genus, simplicial volume) is then sufficient to identify the knot $5_2$ up to mirror image.

\bigskip

In all three theorems, the fact that Seifert pieces inside knot exteriors correspond exactly to connected sums and cablings (see \cite{Bud}) is crucial to the proofs.
Since the simplicial volume does not detect any Seifert pieces in the knot exterior, it can only hope to identify a knot up to connected sums with iterated torus knots and cablings.
This is why it is useful to pair the volume with an invariant that changes under connected sum and cabling, and the genus fits the bill perfectly.
To distinguish $J_n$ from $K_n$ in the Main Theorem, the genus and the volume are no longer sufficient and we must use the monomiality limit $\lambda$ defined in Theorem \ref{thm L2 alex fibered}, which changes under cablings but not under connected sums with iterated torus knots. 

With these results we illustrate 
the strength of the $L^2$-Alexander invariant as a knot invariant.
The value in $t=1$ gives the volume, the degrees in $0^+$ and $+\infty$ yield the genus,  and one can hope to discover much more relevant information at other points. 
Both the exact value of the monomiality limit $\lambda$ and of the leading coefficient $C$ (defined in Theorem \ref{thm genus}) are some leads worthy of further pursuit.

\bigskip

The article is organized as follows:
Section \ref{sec:def} reviews some well-known facts about the $L^2$-Alexander invariant for knots, and can be skimmed by the experienced reader. Section \ref{sec:fib} deals with the behaviour of the $L^2$-Alexander invariant for fibered knots and the detection of the figure-eight knot, and comes from the author's PhD thesis \cite{BAthesis}.
 Section \ref{sec:52} details the detection of the twist knot $5_2$. Section \ref{sec:inf} presents an infinite family of knots detected by the $L^2$-Alexander invariant but not by the pair (genus, volume).

\section*{Acknowledgements}

I would like to thank my PhD advisor J\'er\^ome Dubois, for his teachings and great advice, as well as Rinat Kashaev and Emmanuel Wagner for instructive discussions.

This article is based on work supported by the \textit{Minist\`ere de l'Enseignement Sup\'erieur et de la Recherche} at the Universit\'e Paris Diderot
and by the \textit{Swiss National Science Foundation}, subsidy $200021\_ 162431$, at the Universit\'e de Gen\`eve.

Parts of this article stemmed from ideas inspired by the conference "Manifolds and groups" that took place in Ventotene (LT), Italy.

\section{Preliminaries and previous results}\label{sec:def}

\subsection{Fox calculus}

Here we follow mostly \cite[Chapter 9]{BZ}.

Let $ P = \big\langle g_1, \ldots, g_k \, \big| \, r_1, \ldots r_{l} \big\rangle $ be a presentation of a finitely presented group $G$.
If $w$ is an element of the free group 
$\mathbb{F} [ g_1 , \ldots , g_k ]$
on the generators $g_i$, we let $\overline{w}$ denote the element of $G$ that is the image of $w$ by the composition of the quotient homomorphism (quotient by the normal subgroup $\langle r_j \rangle$ generated by $r_1, \ldots, r_l$) and the implicit group isomorphism between this quotient $Gr(P)$ and $G$. 
To simplify the notations in the sequel, we will often write an element of $G$ $a$ instead of $\overline{a}$ when there is no ambiguity.
We write the corresponding ring morphisms similarly: if $w \in  \C \left [ \mathbb{F} [ g_1 , \ldots , g_k ] \right ]$ then its quotient image is written $\overline{w} \in \C[G] $.

The \textit{Fox derivatives associated to the presentation $P$} are the linear maps \\
$ {\dfrac{\partial}{\partial g_i}\colon \C \left [ \mathbb{F} [ g_1 , \ldots , g_k ] \right ]  \longrightarrow \C \left [ \mathbb{F} [ g_1 , \ldots , g_k ] \right ]}  $ for $i= 1 , \ldots , k$, defined by induction in the following way:

$ \dfrac{\partial}{\partial g_i} (1) = 0 , 
\ \dfrac{\partial}{\partial g_i} (g_j) = \delta_{i,j} , 
\ \dfrac{\partial}{\partial g_i} (g_j^{-1}) = - \delta_{i,j} g_j^{-1}$ (where $\delta_{i,j}$ is $1$ when $i=j$ and $0$ when $i \neq j$)  and 
for all  $u , v \in \mathbb{F} [ g_1 , \ldots , g_n ] , \ \dfrac{\partial}{\partial g_i} (u v) = \dfrac{\partial}{\partial g_i} (u) + u \dfrac{\partial}{\partial g_i} (v)$.

We call 
$F_P = \left ( \overline{ \left ( \dfrac{\partial r_j}{\partial g_i} \right )} \right )_{1 \leqslant i \leqslant k , 1 \leqslant j \leqslant l } \in M_{k,l}(\C \left [ G \right ] )$  the \textit{Fox matrix} of the presentation $P$. 

Let us assume $l =k-1$, i.e. \textbf{$P$ has deficiency one}.
For $i = 1 , \ldots  , k$, 
$F_{P,i} \in M_{k-1,k-1}(\C \left [ G \right ] )$ is defined as the matrix obtained from $F_P$ by deleting its $i$-th row.
%

\subsection{$L^2$-invariants}

Here we follow mostly \cite{Luc02}.

Let $G$ be a countable discrete group (a knot group, for example). In the following, every algebra will be a $\C$-algebra.

Consider the vector space $ \mathbb{C} [G] = \bigoplus_{g \in G} \mathbb{C} g $ (which is also an algebra) and its scalar product:
$$ \left \langle \sum_{g \in G} \lambda_g g , \sum_{g \in G} \mu_g g \right \rangle:= \sum_{g \in G} \lambda_g \overline{\mu_g}. $$

The completion of $ \mathbb{C} [G]$ is $ \ell^2(G):= \left \{ \sum_{g \in G} \lambda_g g \ | \  \lambda_g \in \mathbb{C} , \sum_{g \in G} | \lambda_g |^2 < \infty \right \}$, the Hilbert space of square-summable complex functions on the group $G$.

We denote
$\mathcal{B}( \ell^2(G))$ the algebra of operators on $\ell^2(G)$ that are continuous (or equivalently, bounded) for the operator norm.

To any $h \in G$ we associate a \textit{left-multiplication} 
$L_h\colon \ell^2(G) \rightarrow \ell^2(G)$ defined by
$$ L_h \left ( \sum_{g \in G} \lambda_g g \right ) = \sum_{g \in G} \lambda_g (hg) = \sum_{g \in G} \lambda_{h^{-1}g} g$$
and a \textit{right-multiplication} $R_h\colon \ell^2(G) \rightarrow \ell^2(G)$ defined by
$$ R_h \left ( \sum_{g \in G} \lambda_g g \right ) = \sum_{g \in G} \lambda_g (gh) = \sum_{g \in G} \lambda_{gh^{-1}} g.$$
Both $L_h$ and $R_h$ are isometries, and therefore belong to $\mathcal{B}( \ell^2(G))$.

We will use the same notation for right-multiplications by elements of the complex group algebra $\C [G]$:
$$ R_{ \sum_{i=1}^{k} \lambda_i g_i }:= \sum_{i=1}^{k} \lambda_i R_{g_i} \in \mathcal{B}( \ell^2(G)).$$

We will also use this notation to define a right-multiplication by a matrix $A$ with coefficients in $\C [G]$, $p$ rows and $q$ columns, in the following way:

If $A = \left ( a_{i,j} \right )_{1 \leqslant i \leqslant p, 1 \leqslant j \leqslant q} \in M_{p,q}(\C[G])$, then 
$$R_A:= \left (R_{a_{i,j}}\right )_{1 \leqslant i \leqslant p, 1 \leqslant j \leqslant q} 
\in \mathcal{B}( \ell^2(G)^{\oplus q}; \ell^2(G)^{\oplus p} ).$$

We write $\mathcal{N}(G)$ the algebraic commutant of $ \{ L_g ; g \in G \} $ in $\mathcal{B}( \ell^2(G))$. It will be called the \textit{von Neumann algebra of the group $G$}.

Let us remark that $R_g \in \mathcal{N}(G)$ for all $g$ in $G$.

The \textit{trace} of an element $\phi$ of $\mathcal{N}(G)$ is defined as  $$\mathrm{tr}_{\mathcal{N}(G)}(\phi):= \left \langle \phi (e) , e \right \rangle$$ where $ e $ is the neutral element of $G$.
This induces a trace on the $M_{n,n}(\mathcal{N}(G))$ for $n \geq 1$ by summing up the traces of the diagonal elements. We will write this new trace $\mathrm{tr}_{\mathcal{N}(G)}$ as well.

\begin{remark}\label{rem tr Rg}
If $g\in G$, $g\neq e$, then $\mathrm{tr}_{\mathcal{N}(G)}(R_g) = 0$. If $g=e$, then $R_g = Id_{\ell^2(G)}$ and $\mathrm{tr}_{\mathcal{N}(G)}(Id)=1$.
\end{remark}

\subsection{The Fuglede-Kadison determinant}

Let $G$ be a finitely presented group, $A \in M_{n,n}(\Z G)$ and $R_A\colon \ell^2(G)^n \to \ell^2(G)^n$ the associated operator. 
If $R_A$ is injective, define its \textit{Fuglede-Kadison determinant} as
\begin{equation*}\label{detFK}
{\det}_{\mathcal{N}(G)}(f):= 
\lim_{\epsilon \to 0^+} 
\exp  \left ( \frac{1}{2}
\mathrm{tr}_{\mathcal{N}(G)} \left (
\ln (R_{A}^* R_{A} + \epsilon Id_{\ell^2(G)^n} )
\right )
\right ) \in [0, + \infty [.
\end{equation*}

Here are several properties of the determinant we will use in the rest of this paper (see \cite{Luc02} for more details and proofs).

\begin{prop} \label{prop operations det}
(1) For every nonzero complex number $\lambda$, $\mathrm{det}_{\mathcal{N}(G)} (\lambda Id_U) = | \lambda |$.

(2) For all $f,g$
of the form $R_A, R_B$, with $A \in M_{n,n}(\Z G)$, $B \in M_{k,k}(\Z G)$, $$
\mathrm{det}_{\mathcal{N}(G)} \left (
\begin{pmatrix}
f & 0 \\ 0 & g
\end{pmatrix}
\right ) = \mathrm{det}_{\mathcal{N}(G)}(f) \cdot \mathrm{det}_{\mathcal{N}(G)}(g).$$

(3) Let $f=R_A$ and $g=R_B$ with $A,B \in M_{n,n}(\Z G)$. Assume that $f$ has dense image and $g$ is injective. Then
$$ \mathrm{det}_{\mathcal{N}(G)} (g  \circ f ) =
\mathrm{det}_{\mathcal{N}(G)}(g) \cdot  \mathrm{det}_{\mathcal{N}(G)}(f).$$
\end{prop}

\begin{remark}
If  $f=R_A$ with $A \in M_{n,n}(\Z G)$, then (see \cite[Lemma 1.13]{Luc02}) $f$ is injective if and only if $f$ has dense image.

Therefore, when dealing with \guillemotleft square\guillemotright \ operators,  the property \guillemotleft has dense image\guillemotright \ can be replaced by \guillemotleft is injective\guillemotright \ in the assumptions of Proposition \ref{prop operations det} (3).

\end{remark}
%
%
%
%

\subsection{The $L^2$-Alexander invariant}

Let $K \subset S^3$ be a knot (with orientation),
$\nu K$ an open tubular neighbourhood of $K$,
$M_K = S^3 \setminus \nu K$ the exterior of $K$,
 $G_K = \pi_1(M_K)$ the group of $K$, 
 $P = \big\langle g_1, \ldots, g_k \, \big| \, r_1, \ldots r_{k-1} \big\rangle$ a deficiency one presentation of $G_K$, and
 $\alpha_K\colon G_K \to \Z$ the abelianization, compatible with the orientation of $K$.

For $t>0$ we define the algebra homomorphism:

\begin{equation*}
    \psi_{K,t} \colon\left(
    \begin{aligned}
        \C[G_K] &\longrightarrow  \C[G_K] \\
        \sum_{g \in G_K} c_g \cdot g 
        &\longmapsto 
        \sum_{g \in G_K} c_g \cdot t^{\alpha_{K}(g)} \cdot g
    \end{aligned}
    \right)
\end{equation*}
and we also call $\psi_{K,t}$ its induction to any matrix ring with coefficients in $\C[G_K]$.
One can picture it as a way of  \guillemotleft  tensoring by the abelianization representation\guillemotright .

Consider the operator $R = R_{\psi_{K,t} (F_{P,1})}\colon \ell^2(G)^{k-1} \to \ell^2(G)^{k-1}$. It follows from \cite{FL,L}, that $R$ is always injective and of positive Fuglede-Kadison determinant.

For $f,g \in \mathcal{F}(\R_{>0},\R_{>0})$ two maps, we write $f \ \dot{=} \ g$ if $$\exists m \in \Z, \forall t >0, f(t) = g(t) t^m.$$
Then 
$$\left (t \mapsto
\dfrac{\mathrm{det}_{\mathcal{N}(G_K)}(R_{\psi_{K,t} (F_{P,1})})}
{\max(1,t)^{|\alpha_K(g_1)|-1}}
\right )$$
depends on the presentation $P$ only up to the equivalence relation $\dot{=}$ (see \cite{BA13,DW,LZ06a}). Let us write $\Delta_{K}^{(2)} \in \mathcal{F}(\R_{>0},\R_{>0}) / \dot{=}$ its equivalence class; it is an invariant of the knot $K$ called the \textit{$L^2$-Alexander invariant of $K$}.

We will sometimes abuse the notation by writing $\Delta_{K}^{(2)}(t) \ \dot{=} \ f(t)$: here \\
$\left (t \mapsto \Delta_{K}^{(2)}(t)\right )$ and $ (t \mapsto f(t))$ denote two particular representatives of $\Delta_{K}^{(2)}$ in $\mathcal{F}(\R_{>0},\R_{>0})$.

\begin{exemple} \label{ex L2 unknot}
Let us compute the invariant for the unknot $U$.

\begin{figure}[!h]
\centering
\begin{tikzpicture} [every path/.style={string } , every node/.style={transform shape , knot crossing , inner sep=1.5 pt } ] 
\begin{scope}[xshift=-1cm,scale=1]
	\node[rotate=0] (l) at (-1,0) {};
	\node[rotate=0] (r) at (1, 0 ) { } ;
	
\draw (l.center) .. controls (l.16 north east) and (r.16 north west) .. (r) ;
\draw (r) .. controls (r.32 south east) and (r.32 north east) .. (r.center) ;
\draw (r.center) .. controls (r.16 south west) and (l.16 south east) .. (l) ;
\draw (l) .. controls (l.32 north west) and (l.32 south west) .. (l.center) ;
\end{scope}
\end{tikzpicture}
\caption{A diagram for the unknot} \label{figure doubly twisted}
\end{figure}
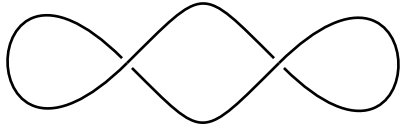

The  \guillemotleft  doubly twisted rubber band\guillemotright \ knot diagram of Figure \ref{figure doubly twisted} gives the Wirtinger presentation $P = \langle g, h | g h^{-1} \rangle$ of the unknot group $G_U$ (which is isomorphic to $\Z$), and the associated Fox matrix is $F_{P} = \begin{pmatrix} 1 \\ -1 \end{pmatrix}$.

Therefore for all $t$ in $\C^*$ , $R_{\psi_{U,t}(F_{P,1})} = -Id\colon \ell^2(G_U) \to \ell^2(G_U)$ has Property $\mathcal{I}$ and $\Delta_{U,P}^{(2)}(t)=1$.
Thus, the invariant for the trivial knot is the constant map equal to $1$.
\end{exemple}

\subsection{Mirror image and reversion}

\

We know how the $L^2$-Alexander invariant behaves under reversion and mirror image, by looking at specific group presentations of the knots in consideration. 

\begin{defi}
Let $K$ be a knot in $S^3$. The knot $K$ comes with a chosen orientation. Its \textit{inverse $-K$} is $K$ with the opposite orientation.
\end{defi}

\begin{prop}\cite[Corollary 4.4]{BA13}
Let $K$ be a knot in $S^3$, and $-K$ its inverse knot. Then
$\Delta^{(2)}_{-K}(t) \ \dot{=} \ \Delta^{(2)}_{K}(1/t)$.
\end{prop}

\begin{defi}
Let $K$ be a knot in $S^3$. The \textit{mirror image of $K$} is the image of $K$ by any planar reflection in $\R^3$. It is written $K^*$. The link $K^*$ does not depend on the plane of reflection up to isotopy.
\end{defi}

\begin{prop}\cite[Proposition 2.26]{BA13}
Let $K$ be a knot in $S^3$, and $K^*$ its mirror image. Then
$\Delta^{(2)}_{K^*}(t) \ \dot{=} \ \Delta^{(2)}_{K}(1/t)$.
\end{prop}

It was proven in \cite{DFLdual} that the $L^2$-Alexander invariant has a similar symmetry property as the classical Alexander polynomial: up to equivalence, composing the $L^2$-Alexander invariant with $(t \mapsto 1/t)$ does not change it. More precisely:

\begin{prop}\cite[Theorem 1.2]{DFLdual} \label{prop dual}
Let $K$ be a knot in $S^3$ and $F$ a representative of the $L^2$-Alexander invariant $\Delta^{(2)}_{K}$. 
Then $F \ \dot{=} \ (t \mapsto F(1/t))$.
\end{prop}

As a consequence of the three previous propositions, the $L^2$-Alexander invariant cannot distinguish a knot from its mirror image or its inverse any more than the classical Alexander polynomial could:

\begin{corollary}
Let $K$ be a knot in $S^3$. Then all knots in the set $\{ K; -K; K^*; -K^* \ \}$ have the same $L^2$-Alexander invariant.
\end{corollary}

As a consequence, we will say that the \textit{$L^2$-Alexander invariant detects a knot $K$} if the only knots with $L^2$-Alexander invariant $\Delta_K^{(2)}$ are  $K, -K, K^*,$ and $-K^*$. Note that for invertible amphicheiral knots, like $U$ and $4_1$, this is the same as the usual stronger meaning of ``detection".

\subsection{Genus, connected sum and cablings}

\

A \textit{Seifert surface} for a knot $K$ is a compact oriented surface $S$ embedded in $S^3$ whose boundary $\partial S$ is equal to $K$ as an oriented $1$-manifold.
The minimal genus $g$ of a Seifert surface spanning a knot $K$ is called \textit{the genus of the knot $K$} and is written $g(K)$.

Let $K$ and $J$ be knots in $S^3$; their \textit{connected sum} $K \sharp J$ (also denoted $S_{J}(K)$ or $S_{K}(J)$) is the knot obtained by removing one small arc on $K$ and on $J$ and joining the four vertices by two other arcs respecting the orientations and forming a single knot. Up to ambient isotopy, the knot $K \sharp J$ does not depend on the choice of the two small arcs.

\begin{prop}\cite[Theorem 5.14]{Rol}\label{prop genus sum}
 The genus of a connected sum $K\sharp J$ of two knots $K$ and $J$ satisfies:
$g(K \sharp J) = g(K) + g(J)$.
\end{prop}

Let $K$ be a knot in $S^3$, $T_K$ an open tubular neighbourhood of $K$ (its core having the same orientation as $K$).
Let $p,q$ be relatively prime integers and let $V$ be a solid torus with a preferred meridian-longitude system. The \textit{$(p,q)$-torus knot $T_{p,q}$} is the knot drawn on $\partial V$ $q$ times in the meridional direction and $p$ times in the longitudinal direction.
Let $h\colon V \to T_K$ be an homeomorphism between the two solid tori preserving orientation and preferred meridian-longitude systems; $h(T_{p,q})$ is a knot in $S^3$ called the \textit{$(p,q)$-cable of $K$} and denoted $C_{p,q}(K)$.

\begin{prop}\cite{shibuya} \label{prop genus cable}
The genus of a cable $C_{p,q}(K)$ of a knot $K$ satisfies:\\
$g(C_{p,q}(K)) = |p| g(K) + (|p|-1)(|q|-1)/2$.
\end{prop}

\subsection{Simplicial volume}

\

We will recall classical definitions of geometry and topology that are necessary to state Thurston-Perelman Geometrization Theorem and to define the simplicial volume.

We assume that a \textit{$3$-manifold} is connected, compact, oriented, and with boundary either empty or a finite union of tori. We will especially consider the examples of exteriors of links in $S^3$.

Let $M$ be a $3$-manifold. We say that $M$ is \textit{irreducible} if every embedded $2$-sphere in $M$ bounds an embedded $3$-ball in $M$. Knot exteriors are all irreducible, and the exterior of a link $L$ is irreducible if and only if the link $L$ is not split.

We say that $M$ is \textit{Seifert} if it admits a foliation by circles (see \cite{epstein}). The exterior of a knot $K$ is Seifert if and only if $K$ is a torus knot, and the exterior of a link $L$ is Seifert if and only if $L$ is a sublink of a link $L' \cup H \cup H'$, where $L'$ is a torus link drawn on a torus $T$ naturally embedded in $S^3$, and $H, H'$ are the cores of the two solid tori bounded by $T$ (see \cite[Proposition 3.3]{Bud}).

We say that $M$ is \textit{hyperbolic} if its interior admits a complete Riemannian metric with sectional curvature $-1$ (following \cite{ratcliffe}). A link $L$ is \textit{hyperbolic} if its exterior is hyperbolic with finite volume. We write $\mathrm{vol}_{hyp}(M)$ (respectively either $\mathrm{vol}_{hyp}(M_L)$ or $\mathrm{vol}_{hyp}(L)$) the hyperbolic volume of $M$ (respectively of $M_L$).

By Thurston-Perelman Geometrization Theorem (see for example \cite[Theorem 1.14]{aschenbrenner20123}), an irreducible manifold $M$ splits along disjoint incompressible tori into pieces that are either Seifert or hyperbolic with finite volume. This family of pieces is unique up to permutation if the number of tori is minimal. Let us note $JSJ(M)$ this minimal set of pieces, the so-called \textit{JSJ-decomposition of $M$}. The \textit{simplicial volume} (or \textit{volume}) $\mathrm{vol}(M)$ of $M$ is defined as the sum of the hyperbolic volumes of the hyperbolic pieces of $JSJ(M)$, and is a topological invariant of $M$. If $L$ is a non-split link in $S^3$, then we write $\mathrm{vol}(L):= \mathrm{vol}(M_L)$.

The $L^2$-Alexander invariant of a knot detects its simplicial volume, following the astonishing theorem of W. Lück and T. Schick that relates the $L^2$-torsion of a $3$-manifold to its simplicial volume (see \cite{LS99}).

\begin{theo}\cite[Theorem 4.6]{Luc02}\label{thm L2 volume}
Let $K$ be a knot in $S^3$. Then $$\Delta^{(2)}_K(1) =
\exp\left (\dfrac{\mathrm{vol}(K)}{6 \pi}\right ).$$
\end{theo}

\subsection{Known results about the $L^2$-Alexander invariant}

We will now state several known properties of the $L^2$-Alexander invariant.

\begin{prop}\label{prop sum}
Let $K\sharp J$ be the connected sum of two knots $K,J$. Then:
\begin{enumerate}
\item $\mathrm{vol}(K\sharp J) = \mathrm{vol}(K) + \mathrm{vol}(J)$,
\item $\Delta^{(2)}_{K\sharp J} = \Delta^{(2)}_{K} \Delta^{(2)}_{J}$.
\end{enumerate}
\end{prop}

The first point comes from the fact that $JSJ(M_{K\sharp J}) = JSJ(M_K) \cup \{ M_L \} \cup JSJ(M_J)$ where $L$ is a three-component keychain link (the connected sum of two Hopf links) of Seifert exterior (see \cite{Bud}). The second point is \cite[Theorem 3.2]{BA13}.

\begin{prop}\label{prop cable}
Let $C_{p,q}(K)$ be the $(p,q)$-cable of a knot $K$. Then:
\begin{enumerate}
\item $\mathrm{vol}(C_{p,q}(K)) = \mathrm{vol}(K)$,
\item $\Delta^{(2)}_{C_{p,q}(K)}(t) \ \dot{=} \ \Delta^{(2)}_{K}(t^p) \max(1,t)^{(|p|-1)(|q|-1)}$.
\end{enumerate}
\end{prop}

The first point comes from the fact that $JSJ\left (M_{C_{p,q}(K)}\right ) = JSJ(M_K) \cup \{ M_L \}$ where $L = T_{p,q} \cup L'$ is a two-component link with $L'$ a trivial knot circling the $p$ strands of $T_{p,q}$, and of Seifert exterior (see \cite{Bud}). The second point is \cite[Theorem 4.3]{BA13}.

We denote by $\mathcal{IT}$ the class of \textit{iterated torus knots}, generated by the trivial knot, the connected sum and all cabling operations. On this class, the $L^2$-Alexander invariant reduces simply to the genus of the knot:

\begin{prop}\cite[Remark 2.39]{BAthesis}\label{prop iter}
Let $K$ be a knot in $S^3$. The following properties are equivalent:
\begin{enumerate}
\item K $\in \mathcal{IT}$,
\item $\mathrm{vol}(K)=0$,
\item $\Delta^{(2)}_{K}(1)=1$,
\item $\Delta^{(2)}_{K} \ \dot{=} \ (t \mapsto \max(1,t)^n)$ for a certain $n \in \N$,
\item $\Delta^{(2)}_{K} \ \dot{=} \ \left (t \mapsto \max(1,t)^{2 g(K)}\right )$.
\end{enumerate}
\end{prop}

\begin{proof}
$(1) \Leftrightarrow (2)$ is due to Gordon (see \cite[Corollary 4.2]{gordon}), 
$(2) \Leftrightarrow (3)$ comes from Theorem \ref{thm L2 volume}, $(5) \Rightarrow (4) \Rightarrow (3)$ is immediate, and $(1) \Rightarrow (5)$ is \cite[Theorem 2.38]{BAthesis} and proven by induction from Propositions \ref{prop sum} and \ref{prop cable}.
\end{proof}

\begin{corollary}(\cite[Theorem 5.1]{BA13}, \cite[Theorem 2.41]{BAthesis}) \label{cor 0 31}
The $L^2$-Alexander invariant detects the unknot $U$ and the trefoil knot $3_1$. More precisely, if $K$ is a knot in $S^3$, then:
\begin{itemize}
\item $(K = U) \Leftrightarrow \left (\Delta^{(2)}_{K} \ \dot{=} \ (t \mapsto 1)\right )$,
\item $(K \in \{3_1, 3_1^*\}) \Leftrightarrow \left (\Delta^{(2)}_{K} \ \dot{=} \ (t \mapsto \max(1,t)^{2})\right )$.
\end{itemize}
\end{corollary}

This corollary follows from Proposition \ref{prop iter} and the well-known facts that the only knot of genus zero is the unknot and the only iterated torus knots of genus one are the trefoil and its mirror image. 

These knots are the only iterated torus knots that can be detected by the $L^2$-Alexander invariant. However, as we shall see in the following sections, the fact that $L^2$-Alexander invariant contains the genus, the simplicial volume and more allows it to detect several knots whose exterior contains hyperbolic pieces.

\section{Fibered knots and detecting $4_1$}\label{sec:fib}

\subsection{The $L^2$-Alexander invariant for fibered knots}

A knot $K$ in $S^3$ is \textit{fibered} if there exists a surface bundle $p\colon M_K \to S^1$ such that the induced group homomorphism $p_* \colon \pi_1(M_K) \to \Z$ is equal to the abelianization
$\alpha_K\colon G_K \to \Z$
 of the knot group.

Recall (see for example \cite[Corollary 5.4]{BZ}) that
the exterior $M_K$ of a fibered knot $K$ of genus $g=g(K)$ is obtained from the product space $\Sigma \times [0;1]$, where $\Sigma= S_{g,1}$ is a compact surface of genus $g$ with connected nonempty boundary, by the identification 
$$(x,0) \sim (h(x),1), \ x \in \Sigma$$
where $h\colon \Sigma \to \Sigma$ is an orientation preserving homeomorphism:
$$ M_K = (\Sigma \times [0;1]) / h.$$
The homeomorphism $h$ is called the \textit{monodromy} associated to $K$.

Furthermore, $G_K = \pi_1(M_K)$ is a semidirect product $G_K = \Z \ltimes G_K'$ where \\
$G_K' \cong \pi_1(\Sigma) \cong \F[a_1, \ldots, a_{2g}]$ is the free group on $2g$ generators, and $G_K$ admits the group presentation
$$P_h = \langle z, a_1, \ldots, a_{2g} |
z a_1 z^{-1} = h_*(a_1), \ldots, 
z a_{2g} z^{-1} = h_*(a_{2g}) \rangle,$$
where $h_*\colon \pi_1(\Sigma) \to \pi_1(\Sigma)$ is the isomorphism induced by $h$.

\begin{lemma} \label{prop id - tA}
Let $G$ be a $3$-manifold group, let $\alpha\colon G \to \Z$ denote a surjective group homomorphism and $H = \mathrm{Ker}(\alpha)$ its kernel. Let $z \in G$ such that $\alpha(z) = 1$ and let $W \in GL_{n}(\Z[H])$ be a matrix invertible over $\Z[H]$.

Let $t>0$, let $ A = R_{W}$ and let $\widetilde{R_z}\colon \ell^2(G)^n \to \ell^2(G)^n$ denote the diagonal operator with $R_z$ at each coefficient.
\begin{enumerate}
\item If $t \in \left ]0;\frac{1}{\|A^{-1}\|_{\infty}}\right [$, then the operator $t \widetilde{R_z} - A$ is invertible and
$$ \det{}_{\NN(G)}\left (t \widetilde{R_z} - A\right ) = 1.$$
\item If $t>\|A\|_{\infty}$, then the operator $t \widetilde{R_z} - A$ is invertible and
$$ \det{}_{\NN(G)}\left (t \widetilde{R_z} - A\right ) = t^n.$$
\end{enumerate}
\end{lemma}

\begin{proof}
Let $t \in \left ]0;\frac{1}{\|A^{-1}\|_{\infty}}\right [$. 
Since $G$ is sofic and the operator $A=R_W$ is invertible, it follows from \cite[Theorem 5]{elek} that $\mathrm{det}_{\NN(G)}(A) = 1$. 
Since
$$t \widetilde{R_z} - A = (-A) \circ (Id - t A^{-1}  \widetilde{R_z}),$$ 
it follows from Proposition \ref{prop operations det} that
we simply have to prove that $Id - t A^{-1}  \widetilde{R_z}$ is invertible of Fuglede-Kadison determinant equal to $1$.

For all $u \in [0;1]$, the operator 
$$S_u := Id - u t A^{-1}  \widetilde{R_z}$$
is invertible, since
$$ \| u t A^{-1}  \widetilde{R_z}\|_{\infty} \leqslant
t \|A^{-1}\|_{\infty} <  1.$$
In particular, $S_0 = Id$ and $S_1 = Id -  t A^{-1}  \widetilde{R_z}$. 

It then follows from \cite[Theorem 1.10(e)]{CFM}  that
$$\det{}_{\NN(G)}(S_1) = \dfrac{\det{}_{\NN(G)}(S_1)}{\det{}_{\NN(G)}(S_0)}
= \exp\left ( \mathrm{Re} \left ( \int_0^1
\mathrm{tr}_{\NN(G)}\left (
S_u^{-1} \circ \dfrac{\partial S_v}{\partial v}|_{v=u} \right )du
\right )\right ).$$

We compute
$$\dfrac{\partial S_v}{\partial v} = 
- t A^{-1}  \widetilde{R_z}$$
and
$$ S_u^{-1} = Id +  ut  A^{-1}  \widetilde{R_z} +
\left (  ut  A^{-1}  \widetilde{R_z}\right )^2 + \ldots$$
Thus the operator
$$S_u^{-1} \circ \dfrac{\partial S_v}{\partial v}|_{v=u} =
- t A^{-1}  \widetilde{R_z} -
u t^2 \left (  A^{-1}  \widetilde{R_z}\right )^2 - \ldots
$$
is an infinite sum of terms $R_{g}$ where $\alpha(g) \geqslant 1$ (since the terms of $A^{-1}$ come from $H$), and its Von Neumann trace $\mathrm{tr}_{\NN(G)}$ is therefore zero (since none of these $g$ can be the neutral element of $G$, see Remark \ref{rem tr Rg}).

Consequently $ \det{}_{\NN(G)}(S_1) = 1$ and the first part of the proposition is proven.

To prove the second part of the proposition, let $t>\|A\|_{\infty}$ and observe that
$$t \widetilde{R_z} - A = 
\left (t \widetilde{R_z}\right ) \circ \left (Id - \frac{1}{t}  \widetilde{R_z}^{-1} A \right );$$
since $t \widetilde{R_z}$ is invertible of Fuglede-Kadison determinant $t^n$, as a consequence of Proposition \ref{prop operations det}, it suffices to prove that the operator $Id - \frac{1}{t}  \widetilde{R_z}^{-1} A $ is invertible of Fuglede-Kadison determinant $1$. This follows from the fact that $t> \|A\|_{\infty}$, from \cite[Theorem 1.10(e)]{CFM}  and from the definition of $A$, similarly as above.
\end{proof}

The following theorem establishes that the $L^2$-Alexander invariant of a knot is monomial in $t$ for extremal values of $t$, and that the span between the degrees of the two monomials is equal to twice the genus of the knot. Compare with the monicity property of the Alexander polynomial for fibered knots (see \cite[Proposition 8.33]{BZ}). This theorem is a variant of \cite[Theorem 8.2]{DFL} for knot exteriors, in the sense that the bound $N$ for monomiality is computed from operator norms associated to the monodromy. It is not yet clear which bound is better between the dilatation factor of the monodromy as in \cite[Theorem 8.2]{DFL} or the bound $N$ constructed from operator norms, and we hope that the following proof will be of relevance to answer this question.

\begin{theo} \label{thm L2 alex fibered}
Let $K$ be a fibered knot of genus $g$, and $$P_h = \langle z, a_1, \ldots, a_{2g} |
z a_1 z^{-1} = h_*(a_1), \ldots, 
z a_{2g} z^{-1} = h_*(a_{2g}) \rangle$$ the presentation of its group $G_K$ associated to the fibration.

There exists a real number $N \geqslant 1$ and a representative $(t \mapsto \delta(t))$ of $\Delta_{K}^{(2)}$
such that 
$$\delta(t) = \left \{ \begin{matrix}
1 &  \mathrm{if} \ t<\frac{1}{N}, \\
t^{2g} &  \mathrm{if} \  t>N. \\
\end{matrix} \right .$$
\end{theo}

The smallest such $N$ will be denoted $\lambda_{\delta}$ or $\lambda_K$ and will be called the \textit{monomiality limit of the function $\delta$}, or the \textit{monomiality limit of the knot $K$}.

\begin{exemple}
From Proposition \ref{prop iter}, if $K$ is an iterated torus knot, then $\lambda_K =1$.
\end{exemple}

\begin{remark}\label{rem cont}
Since the $L^2$-Alexander invariant is continuous (see \cite[Theorem 1.2 (1)]{Liu}), if $K$ is a fibered knot with $\Delta_{K}^{(2)}(1) > 1$ (ie $\mathrm{vol}(K)>0$), then $\lambda_K >1$.
\end{remark}

\begin{exemple} \label{ex 41 fibered}
If $K = 4_1$ is the figure-eight knot, then there exists a $T>1$ such that
$$\Delta_K^{(2)}(t) = \left \{ \begin{matrix}
1 &  \mathrm{if} \ t<\frac{1}{T}, \\
 \exp\left (\dfrac{\mathrm{vol}(4_1)}{6 \pi}\right ) \approx 1.113 & \mathrm{if} \  t=1, \\
t^2 &  \mathrm{if} \  t>T. \\
\end{matrix} \right .$$
At the time of writing, the smallest known $T$ satisfying the above is $$T = \dfrac{3 + \sqrt{5}}{2} \approx 2.618,$$
the dilatation factor associated to the monodromy.
Thus  $\lambda_{4_1} \in ]1; 2.618..]$.
\end{exemple}

\begin{proof}
Let $K$ be a fibered knot of genus $g$, $\Sigma$ the associated fibre, a once-punctured surface of genus $g$. The group $\pi_1(\Sigma)$ is a free group on $2g$ elements $a_1, \ldots a_{2g}$, and $G_K$ has the presentation 
$P_h = \langle z, a_1, \ldots, a_{2g} | z a_i z^{-1} = W_i(a_j) \rangle$
where $W_i(a_j) = h_*(a_j)$ is a word in the letters $a_j$.

The abelianization $\alpha_K\colon G_{K} \to \Z$ sends $z$ to $1$ and the $a_i$ to $0$.

Since the presentation $P_h$ is of deficiency one, we can compute the $L^2$-Alexander invariant of $K$ from the Fox matrix $F_{P_h}$.

The Fox matrix $F_{P_h}$ is written
$$F_{P_h} =
\begin{pmatrix} 
1 - z a_1 z^{-1} & 1 - z a_2 z^{-1} &  \ldots & 1 - z a_{n} z^{-1} \\
z - w_{1,1} & -w_{1,2} & \ldots & -w_{1,n} \\
 -w_{2,1} & z - w_{2,2} & \ldots & -w_{2,n} \\
\vdots & \vdots & \ddots & \vdots \\
-w_{n,1} & -w_{n,2} & \ldots & z - w_{n,n} 
 \end{pmatrix},$$
 where $n=2g$ and $w_{i,j} = \dfrac{\partial W_j(a_1, \ldots,a_{2g})}{\partial a_i} \in \Z[G_{K}]$ is a linear combination of words on the generators $a_1, \ldots, a_{2g}$; these words are all sent to zero by the abelianization $\alpha_K$.

Let $W$ denote the jacobian matrix
$$W = \begin{pmatrix}
w_{1,1} & \ldots & w_{1,n} \\
\vdots & \ddots &\vdots \\
w_{n,1} & \ldots & w_{n,n}
\end{pmatrix} = 
\left (\dfrac{\partial(h_*(a_j))}{\partial a_i}\right )_{1 \leqslant i,j \leqslant n}
$$
and $A = R_{\psi_{K,t}(W)}$ the associated operator.


For all $t>0$, the operator $R_{\psi_{K,t}(F_{P_h,1})}$ is of the form
$$R_{\psi_{K,t}(F_{P_h,1})} = 
t \widetilde{R_z} - A
=\begin{pmatrix} 
t R_z - A_{1,1} & -A_{1,2} & \ldots & -A_{1,n} \\
 -A_{2,1} & t R_z - A_{2,2} & \ldots & -A_{2,n} \\
\vdots & \vdots & \ddots & \vdots \\
-A_{n,1} & -A_{n,2} & \ldots & t R_{z} - A_{n,n}
 \end{pmatrix},$$ where $n=2g$,  $A_{i,j} =
 R_{\psi_{K,t}(w_{i,j})}$ for all $i,j$, and $\widetilde{R_z}$ is the diagonal operator with $R_z$ at each coefficient. 

Since $h_*\colon \F[a_1, \ldots, a_n] \to \F[a_1, \ldots, a_n]$ is an isomorphism, $(h_*(a_1), \ldots, h_*(a_n))$ is a basis for the free group
$\F[a_1, \ldots, a_n]$ and it follows from \cite{birman} that 
the jacobian matrix 
$$W = 
\left (\dfrac{\partial(h_*(a_j))}{\partial a_i}\right )_{1 \leqslant i,j \leqslant n}$$ is invertible over $\Z[\F[a_1, \ldots, a_n]]$.
It follows from Lemma \ref{prop id - tA} 
 that if \\
 $t \in ]0;\frac{1}{\|A^{-1}\|_{\infty}}[ \cup ]\|A\|_{\infty};+ \infty[$, then the operator $R_{\psi_{K,t}(F_{P_h,1})} = t \widetilde{R_z} - A$ is invertible and

\begin{enumerate}
\item if $t \in ]0;\frac{1}{\|A^{-1}\|_{\infty}}[$, then 
$$ \det{}_{\NN(G)}\left (R_{\psi_{K,t}(F_{P_h,1})} \right ) = 1,$$
\item if $t>\|A\|_{\infty}$, then 
$$ \det{}_{\NN(G)}\left (R_{\psi_{K,t}(F_{P_h,1})} \right ) = t^n.$$
\end{enumerate}

Let $N = \min\left (\|A\|_{\infty},\|A^{-1}\|_{\infty}\right )$.
It follows  from Proposition \ref{prop dual}
 that  for all $t \in ]0;\frac{1}{N}[ \cup ]N;+ \infty[$:
$$\Delta_K^{(2)}(t) \ \dot{=} \ 
\dfrac{\det{}_{\NN(G)}\left (R_{\psi_{K,t}(F_{P_h,1})}\right )}{\max(1,t)^{|\alpha_K(z)|-1}} = 
\det{}_{\NN(G)}\left (R_{\psi_{K,t}(F_{P_h,1})}\right ) =
\left \{ \begin{matrix}
1 &  \mathrm{if} \ t<\frac{1}{N}, \\
t^{2g} &  \mathrm{if} \  t>N. \\
\end{matrix} \right .$$

%
\end{proof}
 

\begin{remark}
Consider $K$ an hyperbolic fibered knot, with associated pseudo-Anosov monodromy $h$. Let us denote $\mathrm{ent}(h)$ its entropy, i.e. the logarithm of the dilatation factor (see \cite{FLP} for details). Let $f$ be the representative of $\Delta_{K}^{(2)}$ such that $f(t) = 1$ for $t \in ]0;\lambda_K^{-1}[$ and $f(t)=t^{2 g(K)}$ for $t \in ]\lambda_K; + \infty[$.

From \cite[Theorem 1.2]{Liu}, $f$ is multiplicatively convex and thus $$\exp\left (\dfrac{\mathrm{vol}(M_K)}{6 \pi}
\right )=
f(1) \leqslant \sqrt{f(\lambda_K^{-1})f(\lambda_K)} = 
\sqrt{\lambda_K^{2 g(K)}} = \lambda_K ^{g(K)}$$
by Theorem \ref{thm L2 volume}.

Hence $$ \ln (\lambda_K) \geqslant \dfrac{\mathrm{vol}(M_K)}{(3\pi) (2 g(K))}$$
which, combined with the fact that $\mathrm{ent}(h) \geqslant \ln(\lambda_K)$ (from \cite[Theorem 8.2]{DFL}), gives us a slightly weaker version of Kojima-Macshane's inequality 
$$ \mathrm{ent}(h) \geqslant \dfrac{\mathrm{vol}(M_K)}{(3\pi) (2 g(K)- 1)}$$
(see \cite[Theorem 1]{kojima}).

It is quite remarkable to find such similar inequalities obtained with such different methods.

Furthermore, note that finding that $2 g(K) \ln (\lambda_K)$ is actually smaller than \\
$(2 g(K) -1) \mathrm{ent}(h)$ would give a stronger upper bound on the volume than with \cite[Theorem 1]{kojima}.

Finally, note that since $\lambda_K \leqslant \min\left (\|A\|_{\infty},\|A^{-1}\|_{\infty}\right )$ where $A$ is the Fox jacobian of the monodromy $h$ (see Theorem \ref{thm L2 alex fibered} above), the multiplicative convexity gives us a lower bound on the operator norm $\|A\|_{\infty}$ depending on the volume.
\end{remark}

\subsection{The $L^2$-Alexander invariant detects the figure-eight knot}

\begin{figure}[!h]
\centering
\includegraphics[scale=0.2]{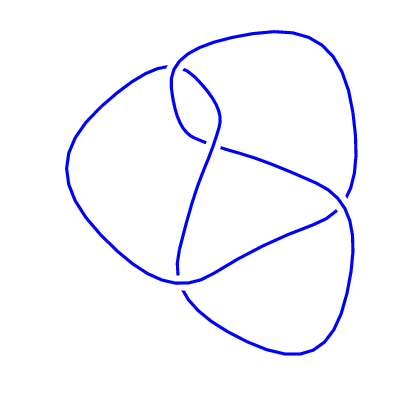}
\caption{The figure-eight knot $4_1$} \label{figure 41}
\end{figure}

We can now prove that the $L^2$-Alexander invariant detects the figure-eight knot $4_1$.

\begin{theo}\label{thm 41}
Let $K$ be a knot in $S^3$ such that its $L^2$-Alexander invariant $\Delta_K^{(2)}$ is of the form:
$$\Delta_K^{(2)}(t) \ \dot{=} \ \left \{ \begin{matrix}
1 &  \mathrm{if} \ t<\frac{1}{T} \\
 \exp\left (\dfrac{\mathrm{vol}(4_1)}{6 \pi}\right ) \approx 1.113 & \mathrm{if} \  t=1 \\
t^2 &  \mathrm{if} \  t>T \\
\end{matrix} \right .$$
for a certain $T \in [1,+ \infty[$.
Then $K$ is the figure-eight knot $4_1$.
\end{theo}
\begin{proof}
Let $K$ be a knot satisfying the assumptions of the theorem. It follows from Theorem \ref{thm L2 volume} that $\mathrm{vol}(M_K) = \mathrm{vol}(4_1)$.

The JSJ-decomposition of $M_K$ contains Seifert-fibered and hyperbolic pieces; these pieces are all sub-manifolds of $S^3$ whose boundary are non-empty finite unions of tori.

A compact hyperbolic $3$-manifold with toroidal boundary and at least three boundary components has volume at least three times the volume of a regular ideal terahedron (see \cite{adams1988volumes}), thus it has volume  greater than $3$. Moreover, a compact hyperbolic $3$-manifold with toroidal boundary and two boundary components has volume at least $3.66...$ (the volume of the Whitehead link) by \cite{agol2010minimal}.
Since the simplicial volume $\mathrm{vol}(M_K)$ of $M_K$ is equal to $\mathrm{vol}(4_1) = 2.029...$, it is smaller than $3$ and since $\mathrm{vol}(M_K)$ is equal to the sum of the simplicial volumes of the JSJ-pieces of $M_K$, we conclude that all the hyperbolic pieces in the JSJ decomposition of $M_K$ have exactly one boundary component (which is a torus).

A compact hyperbolic $3$-manifold with one toroidal boundary component has volume at least $\mathrm{vol}(4_1) = 2.029...$; among these manifolds, only the exterior of the figure-eight knot $M_{4_1}$ and its sibling $M'$ (which can be described as the $(-5/1)$-Dehn filling on the Whitehead link)
 have volume equal to this number (see \cite{cao}).

This implies that the JSJ decomposition of $M_K$ has exactly one hyperbolic piece $N$, which is homeomorphic to $M_{4_1}$ or $M'$; the other pieces are Seifert-fibered manifolds that we  denote by $S_j$.

The manifold $N$ is a  compact sub-manifold of $S^3$ with boundary a single torus, thus $N$ is the exterior $M_{K'}$ of a knot $K'$ (see for example \cite[Proposition 2.2]{Bud}). Since the first homology group of the manifold $M'$ is $$H_1(M'; \Z) = \Z \oplus \Z/5\Z,$$
the manifold $M'$ cannot be the exterior of a knot in $S^3$ and therefore $N = M_{4_1}$.

Since the manifold $M_K$ has a JSJ decomposition composed of $M_{4_1}$ and Seifert-fibered manifolds $S_j$,
it follows from \cite[Theorem 4.18]{Bud} that 
$K$ is obtained from $4_1$ by a finite number of \begin{itemize}
\item cablings,
\item connected sums with an iterated torus knot.
\end{itemize}
The knot $4_1$ is fibered, all iterated torus knots are fibered, the connected sum of two fibered knots is fibered and all cablings of a fibered knot are fibered (see for example \cite[p. 326]{Rol} and \cite{stallings}). Thus the knot $K$ is fibered.

It follows from Theorem \ref{thm L2 alex fibered} and the assumptions on $K$ that $g(K) = 1$. Since the only hyperbolic fibered knot of genus $1$ is $4_1$ (see \cite[Proposition 5.14]{BZ}) we conclude that $K$ is the figure-eight knot $4_1$.

\end{proof}

Note that this proof was found before the following Theorem \ref{thm genus}, this is why Theorem \ref{thm L2 alex fibered} is used to detect the value of the genus at the end of the proof. We chose to leave the proof like this since the reasoning may be useful for possible generalisations about \textit{fibered} manifolds where Theorem \ref{thm genus} may not apply.

\section{Detecting the twist knot $5_2$}\label{sec:52}

\begin{figure}[!h]
\centering
\includegraphics[scale=0.2]{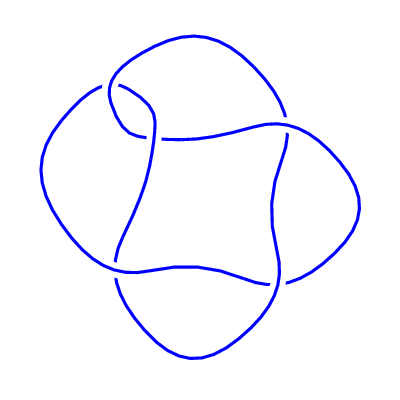}
\caption{The non-fibered twist knot $5_2$} \label{figure 52}
\end{figure}

It has been recently proven that the $L^2$-Alexander invariant of a knot always detects the genus, not only for iterated torus knots (Proposition \ref{prop iter}) or for fibered knots (Theorem \ref{thm L2 alex fibered}). More precisely:

\begin{theo}(\cite{FL},\cite[Theorem 1.2]{Liu})\label{thm genus}

Let $K$ be a knot in $S^3$, then any representative $(t \mapsto f(t))$ of its $L^2$-Alexander invariant $\Delta_K^{(2)}$ satisfies :
\begin{itemize}
\item $f(t) \sim_{t \to 0^+} C t^{m} $,
\item $f(t) \sim_{t \to +\infty} C t^{m + 2 g(K)} $,
\end{itemize}
where $C \in \left [1, \exp\left (\dfrac{\mathrm{vol}(K)}{6 \pi}\right )\right ]$ depends only on the knot, and $m \in \Z$.
In particular, the $L^2$-Alexander invariant of a knot detects its genus.
\end{theo}

We can now prove that the pair (genus, volume), and thus the $L^2$-Alexander invariant, detects the knot $5_2$ (even if we  know the exact value of $\Delta_{5_2}^{(2)}(t)$ only for $t=1$ !).

\begin{theo}\label{thm 52}
Let $K$ be a knot in $S^3$ such that its $L^2$-Alexander invariant $\Delta_K^{(2)}$ has a representative $f(t)$ satisfying :
\begin{itemize}
\item $f(t) \sim_{t \to 0^+} C $
\item $f(1) =  \exp\left (\dfrac{\mathrm{vol}(5_2)}{6 \pi}\right ) \approx 1.162 $
\item $f(t) \sim_{t \to +\infty} C t^2  $
\end{itemize}
for a certain $C \in [1,1.162[$.
Then $K$ is the knot $5_2$.
\end{theo}

\begin{remark}
We can actually prove that $C=1$ for the knot $5_2$ and for all other twist knots, using a particular group presentation (see Remark 3.10 and Section B.2.4 in \cite{BAthesis}). Various computations have shown $C$ to be equal to $1$ for several other knots. We are currently studying the behaviour of this constant $C$ in more detail.
\end{remark}

\begin{proof}
Let $K$ be a knot satisfying the assumptions of the theorem. It follows from Theorem \ref{thm L2 volume} that $\mathrm{vol}(M_K) = \mathrm{vol}(5_2)$.

The JSJ-decomposition of $M_K$ contains Seifert-fibered and hyperbolic pieces; these pieces are all sub-manifolds of $S^3$ whose boundary are non-empty finite unions of tori.
From there we can use the same arguments as in the proof of Theorem \ref{thm 41} to conclude
%
%
that the JSJ decomposition of $M_K$ has exactly one hyperbolic piece $N$, which has hyperbolic volume $2.828$; the other pieces are Seifert-fibered manifolds that we  denote by $S_j$.

By \cite{GMM}, $N$ is thus one of the three hyperbolic manifolds $m015$ (which is the exterior of the knot $5_2$), $m016$ (which is the exterior of the knot $K12n242$), and $m017$.
The manifold $N$ is a  compact sub-manifold of $S^3$ with boundary a single torus, thus $N$ is the exterior $M_{K'}$ of a knot $K'$ (see for example \cite[Proposition 2.2]{Bud}).
Since the first homology group of the manifold $m017$ is $$H_1(m017; \Z) = \Z \oplus \Z/7\Z,$$
the manifold $m017$ cannot be the exterior of a knot in $S^3$ and therefore $K'$ is either $5_2$ or $K12n242$.

Since the manifold $M_K$ has a JSJ decomposition composed of $M_{K'}$ and Seifert-fibered manifolds $S_j$,
it follows from \cite[Theorem 4.18]{Bud} that 
$K$ is obtained from $K'$ by a finite number of \begin{itemize}
\item cablings,
\item connected sums with an iterated torus knot.
\end{itemize}
All these operations strictly increase the genus of the knot (see Propositions \ref{prop genus sum} and \ref{prop genus cable}), except for the trivial cablings (of the type $(\pm 1, m)$) and for connected sums with the trivial knot.
From the assumptions of the theorem and Theorem \ref{thm genus}, $K$ is of genus one. Since $5_2$ is of genus one and $K12n242$ is of genus $5$ (see the website \textit{Knotinfo} for example), the only possible $K$ is $5_2$.
\end{proof}

\begin{remark}
Note that $K12n242$ is fibered and $5_2$ is not, but we cannot yet distinguish for example $K12n242$ from $5_2 \sharp T_{2,9}$, since they both have simplicial volume $2.828...$ and genus $5$ (and the leading coefficients $C$ of their $L^2$-Alexander invariants are both equal to $1$). Knowing if non-fibered knots can have an $L^2$-Alexander invariant of the form of Theorem \ref{thm L2 alex fibered}
(i.e. if a non-fibered knot can have a finite monomiality limit)
 would help distinguish such knots.
\end{remark}

\section{An infinite family of pairs of knots distinguished by the $L^2$-Alexander invariant and not by the pair (genus,volume)} \label{sec:inf}

In the four previous examples ($U$, $3_1$, $4_1$ and $5_2$), we proved that the $L^2$-Alexander invariant detects each of these knots, and the proofs could be roughly separated in two steps :
\begin{enumerate}
\item The $L^2$-Alexander invariant contains the simplicial volume (Theorem \ref{thm L2 volume}) and the genus (Proposition \ref{prop iter}, Theorem \ref{thm L2 alex fibered}, Theorem \ref{thm genus}).
\item The pair (simplicial volume, genus) detects each of the knots $U$, $3_1$, $4_1$ and $5_2$.
\end{enumerate}
The $L^2$-Alexander invariant being a \textit{continuous} function (thanks to \cite[Theorem 1.2]{Liu}) containing both the genus and the volume associated to a knot makes it already a powerful and interesting object.
However, we will now prove that the $L^2$-Alexander invariant contains \textit{strictly more information} than the pair (genus, volume), which is illustrated by further knot detection properties, as explained in the following theorem.

\begin{theo}\label{thm inf}
There exists two infinite families of fibered knots $(J_n)_{n  \in \N}$ and $(K_n)_{n  \in \N}$ such that for all $n \in \N$ :
\begin{enumerate}
\item $g(J_n) = g(K_n)$,
\item $\mathrm{vol}(J_n) = \mathrm{vol}(K_n)$,
\item $\Delta_{J_n}^{(2)} \neq \Delta_{K_n}^{(2)}$,
\item $K_n$ is the only knot (up to mirror image and reversion, as always) with $L^2$-Alexander invariant equal to $\Delta_{K_n}^{(2)}$.
\end{enumerate}
\end{theo}

Note that the last point proves that there exists an infinite family of knots detected by the $L^2$-Alexander invariant.

\begin{proof}
Let $n \in \N$ and let $p$ be the $n$-th prime number (starting at $p=2$ for $n=0$).

Let $J_n = 4_1 \sharp (3_1)^{\sharp (p-1)}$ be the connected sum of the figure-eight knot and of $p-1$ copies of the trefoil knot. 
Let $K_n = C_{p,1}(4_1)$ be the $(p,1)$-cable of the figure-eight knot.

From Propositions \ref{prop genus sum} and \ref{prop genus cable}, we find that $g(J_n) =g(K_n) = p$, which proves $(1)$.

From Propositions \ref{prop sum} (1) and \ref{prop cable} (1), we have 
$\mathrm{vol}(J_n) = \mathrm{vol}(K_n) = \mathrm{vol}(4_1)$, which proves $(2)$.

Let us prove $(3)$.
Let $T = (t \mapsto \max(1,t)^2)$ be a representative of $\Delta_{3_1}^{(2)}$ and $F$ be a representative of $\Delta_{4_1}^{(2)}$ such that $F(t) \sim_{t \to 0^+} 1$ and $F(t) \sim_{t \to + \infty} t^2$.
Then, by Proposition \ref{prop sum} (2),
$\Delta_{J_n}^{(2)} \ \dot{=} \ (t \mapsto F(t)T(t)^{p-1})$ and
by Proposition \ref{prop cable} (2), $\Delta_{K_n}^{(2)} \ \dot{=} \ (t \mapsto F(t^p))$.
From Example \ref{ex 41 fibered}, the monomiality limit for the figure-eight knot (and the map $F$) is $\lambda_F > 1$, whereas the one for the trefoil knot is $\lambda_T = 1$.
Hence $\lambda_{J_n} = \max(\lambda_F, \lambda_T) = \lambda_F$ is different from $\lambda_{K_n} = \lambda_F ^{1/p}$.
Thus the $L^2$-Alexander invariants $\Delta_{J_n}^{(2)}$ and $\Delta_{K_n}^{(2)}$ are different, which proves $(3)$.

Finally, let us prove $(4)$.
Let $K$ be a knot such that $\Delta_{K}^{(2)} = \Delta_{K_n}^{(2)}$. Thus, from what precedes, $g(K) = p$, $\mathrm{vol}(K) = \mathrm{vol}(4_1)$ and $\lambda_K = \lambda_F ^{1/p}$.
Since $\mathrm{vol}(K) = \mathrm{vol}(4_1)$, $K$ is obtained from $4_1$ by cablings and connected sums with iterated torus knots (see the proof of Theorem \ref{thm 41}). Thus
$$K = (S_{I_1} \circ C_{a_1,b_1} \circ S_{I_2} \circ \ldots \circ C_{a_k,b_k} \circ S_{I_{k+1}} ) (4_1)$$ where the $I_j$ are (possibly trivial) iterated torus knots and the $(a_j,b_j)$ are pairs of relatively prime numbers.
Therefore $\lambda_K = \lambda_F^{\frac{1}{|a_1 . \ldots . a_k|}}$, hence $p = |a_1 . \ldots a_k|$. But since $p$ is prime, this means that only one of the $a_i$ is equal to $\pm p$ and the other $a_i$ are equal to $\pm 1$. Up to summing and reversing the orientation of some of the iterated torus knots $I_i$, we can thus assume that 
$$K = (S_{I_1} \circ C_{\pm p,b} \circ S_{I_2})(4_1)$$ where $b$ is prime with $p$.
Therefore, from Propositions \ref{prop genus sum} and \ref{prop genus cable}, $$g(K) = g(I_1) + p(1 + g(I_2)) + (p-1)(|b|-1)/2.$$ Since we know that $g(K)=p$, we deduce that $g(I_1) = g(I_2) =0$ and $b=\pm 1$. Thus $K = C(\pm p,\pm 1)(4_1)$, i.e. 
$K \in \{K_n, -K_n, K_n^*, -K_n^*\}$, which proves $(4)$.

\end{proof}


\bibliographystyle{plain}

\bibliography{bibliothese}

\end{document}